\documentclass[letterpaper,12pt]{amsart}
\usepackage{color}
\pdfoutput=1
\usepackage{subfigure}

\title{Geodesics in CAT(0) Cubical Complexes}

\author{Federico Ardila}
\address{Department of Mathematics \\
San Francisco State University, San Francisco, CA 94132}
\email{federico@math.sfsu.edu}
\author{Megan Owen}
\address{Department of Mathematics \\
University of California Berkeley, Berkeley, CA 94720}
\email{maowen@berkeley.edu}
\author{Seth Sullivant}
\address{Department of Mathematics \\
North Carolina State University, Raleigh, NC 27695}
\email{smsulli2@ncsu.edu}

\thanks{Ardila, Owen, and Sullivant were partially supported by the U.S. National Science Foundation under awards DMS-0635449,  DMS-0801075, DMS-0840795, DMS-0954865, and DMS-0956178.  Ardila was also partially supported by an SFSU Presidential Award and Sullivant was also partially supported by the David and Lucille Packard Foundation.
}

\date{}

\usepackage{latexsym,array,delarray,amsthm,amssymb,graphicx}

\theoremstyle{plain}
\newtheorem{thm}{Theorem}[section]
\newtheorem{lemma}[thm]{Lemma}
\newtheorem{prop}[thm]{Proposition}
\newtheorem{cor}[thm]{Corollary}

\theoremstyle{definition}
\newtheorem{defn}[thm]{Definition}
\newtheorem{ex}[thm]{Example}
\newtheorem{pr}[thm]{Problem}
\newtheorem{alg}[thm]{Algorithm}

\newtheorem{remark}[thm]{Remark}

\theoremstyle{remark}

\newcommand{\zz}{\mathbb{Z}}

\newcommand{\rr}{\mathbb{R}}

\newcommand{\calc}{\mathcal{C}}

\newcommand{\calH}{\mathcal{H}}

\newcommand{\calo}{\mathcal{O}}

\newcommand{\ind}{\mbox{$\perp \kern-5.5pt \perp$}}

\newcommand{\bs}{\backslash}

\providecommand{\norm}[1]{\lVert#1\rVert}

\numberwithin{equation}{section}

\begin{document}

\maketitle

\begin{abstract}
We describe an algorithm to compute the geodesics in an arbitrary CAT(0) cubical complex.  
A key tool is a correspondence between cubical complexes of global non-positive curvature and \emph{posets with inconsistent pairs}. This correspondence also gives an explicit realization of such a complex as the state complex of a reconfigurable system, and a way to embed any interval in the integer lattice cubing of its dimension.

\end{abstract}


\section{Introduction}

A \emph{cubical complex} is a polyhedral complex where all cells are cubes and all attaching maps are injective. Informally speaking, it is just like a simplicial complex, except that the cells are cubes instead of simplices. Every cubical complex has an intrinsic metric induced by the Euclidean $L^2$ metric on each cube.  A polyhedral complex is \emph{CAT(0)} if and only if it is globally non-positively curved. 
This implies that there is a unique local geodesic between any two points.  CAT(0) cubical complexes make frequent appearances in mathematics and its applications, for instance in geometric group theory, in the theory of reconfigurable systems, and in phylogenetics.  The main goal of this paper is to describe an algorithm for computing geodesics in CAT(0) cubical complexes.

A prototypical example of CAT(0) cubical complexes comes from ``reconfigurable systems", a broad family of systems which change according to local rules. Examples include robotic motion planning, the motion of non-colliding particles in a graph, and phylogenetic tree mutation, among many others.  In many reconfigurable systems, the \emph{parameter space} of all possible positions of the system naturally takes the shape of a CAT(0) cubical complex $X$ \cite{GhristPeterson07}. Finding geodesics in $X$ is equivalent to finding the optimal way to get the system from one position to another one under this metric.

Our algorithm appears in Section \ref{sec:geod}; before we can describe it, we need a number of preliminary steps.  First, in Section \ref{sec:combo}, we develop a new combinatorial description of CAT(0) cubical complexes in terms of distributive lattices of partially ordered sets (posets) with inconsistent pairs.  In Section \ref{sec:applications} we give two applications of this description. In Section \ref{sec:interval} we develop the notion of an interval between two cubes in a CAT(0) cubical complex.  The interval contains the geodesic between any point in the first cube and any point in the second cube.  The interval is a new CAT(0) cubical complex which corresponds to an ordinary distributive lattice (no inconsistent pairs).  This interpretation allows us to prove the conjecture \cite{NSW} that a $d$-dimensional interval embeds in the lattice cubing $\zz^d$. In Section \ref{sec:realizable} we review reconfigurable systems and their state complexes, and show that our combinatorial description of a CAT(0) cubical complex provides a way to realize it as such a state complex.

In Section \ref{sec:combin.geodesics} we introduce \emph{valid cube sequences}, the paths of cubes that a geodesic can pass through. In Section \ref{sec:zero}, we describe locally the geodesics passing through a given set of cubes. In Section \ref{sec:noShortcut} we give a criterion to determine whether the geodesic through a given sequence of cubes is the global geodesic; and if it is not, to detect a better sequence of cubes for it to pass through. This description of geodesics and check for improvement generalizes the main criteria from \cite{OwenProvan11} for determining geodesics in phylogenetic tree space.  This provides a characterization of the geodesic between two points. 

After reviewing touring problems in Section \ref{sec:touring}, we describe our algorithm in Section \ref{sec:geod}. We compute the shortest path through a given sequence of cubes, translating this into a fixed order touring problem, which can be solved in polynomial time using semidefinite programming \cite{PolishchukMitchell05}.   We then use the criteria of Sections \ref{sec:zero} and \ref{sec:noShortcut} to determine whether this is the global geodesic, and if it is not, to improve it.

Unlike in the case of tree space, where the resulting touring problem can be solved in linear time \cite{Owen09} and all break points have constructible coordinates (obtained by a sequence of quadratic field extensions), the touring problem for general CAT(0) cubical complexes has intrinsic algebraic complexity.  In particular, geodesics in general CAT(0) complexes can have break points whose coordinates have nonsolvable Galois groups, implying that the iterative algorithms are probably essential for this problem, since there is no exact ``simple" formula for the geodesic.  We explore this algebraic complexity in Section \ref{sec:algebra}.   

\smallskip

\noindent \textbf{Related results in the literature.}
The problem of finding the shortest path between two points in some Euclidean region is very well studied in two dimensions.  In a number of general situations in two dimensions, the shortest path algorithm is polynomial.  For example, the Euclidean shortest path between polygonal obstacles in the plane can be computed in time $O(n \log n)$, where $n$ is the number of vertices in the obstacle polygons \cite{HershbergerSuri97}.  For a detailed survey, see \cite{Mitchell00}.  Recently, Chepoi and Maftuleac \cite{ChepoiMaftuleac10} gave a polynomial algorithm for computing the shortest path through a CAT(0) rectangular complex, in which each cell is 2--dimensional. 

In three or more dimensions, computing Euclidean shortest paths through a region with obstacles is NP-complete \cite{CannyReif87}.  Indeed, this problem is NP-complete in 3 dimensions even when the obstacles are restricted to being disjoint axis-aligned boxes \cite{MitchellSharir04}.  Ghrist and Lavalle \cite{GhristLavalle06} observed, however, that no example in \cite{MitchellSharir04} that is NP-hard is a CAT(0) space.  

All known polynomiality results in higher dimensions are for CAT(0) spaces.  For example, the shortest path through a \emph{cube-curve}, a sequence of face-connected 3-dimensional cubes that does not intersect itself, can be computed in linear time \cite{LiKlette06}. 
For general CAT(0) cube complexes, edge geodesics (geodesics along the edges of the cubes) are well understood, and easily constructed in terms of Niblo and Reeves's \emph{normal cube paths} \cite{NibloReeves98}. These play an important role in geometric group theory, where they are used to study groups acting on CAT(0) cube complexes. Much less is known about geodesics in the $L^2$-metric. An important, well understood example is the space of phylogenetic trees of Billera, Holmes, and Vogtmann \cite{BHV01}, where geodesics can be computed in polynomial time \cite{OwenProvan11}.  This result was generalized to any \emph{CAT(0) orthant space}; \emph{i.e.}, to any set of orthants arranged around a common origin that is CAT(0) \cite{MillerOwenProvan10}. To the best of our knowledge, our algorithm is the first to compute geodesics in an arbitrary CAT(0) cubical complex.

\medskip

\noindent \textbf{Acknowledgments.} We would like to acknowledge Rick Scott for valuable discussions on CAT(0) cubical complexes, and for an excellent talk at the Bay Area Discrete Mathematics Day in 2007 which first inspired Theorems \ref{th:poset} and \ref{th:embedding}.


\section{Combinatorial Geometry of CAT(0) Cubical Complexes}\label{sec:combo}

We begin by defining CAT(0) spaces, the spaces of global non-positive curvature that we are interested in. Let $X$ be a metric space where there is a unique geodesic (shortest) path between any two points. Consider a triangle $T$ in $X$ of side lengths $a,b,c$, and build a comparison triangle $T'$ with the same lengths in Euclidean plane. Consider a chord of length $d$ in $T$ which connects two points on the boundary of $T$; there is a corresponding comparison chord in $T'$, say of length $d'$. If for every triangle $T$ in $X$ and every chord in $T$ we have $d \leq d'$, we say that $X$ is \emph{CAT(0)}.

Testing whether a general metric space is CAT(0) is quite subtle. However, Gromov \cite{Gromov} proved that this is easier if the space is a cubical complex. In a cubical complex, the link of any vertex is a simplicial complex. We say that a simplicial complex $\Delta$ is \emph{flag} if it has no empty simplices; \emph{i.e.}, if any $d+1$ vertices which are pairwise connected by edges of $\Delta$ form a $d$-simplex in $\Delta$.

\begin{thm}[Gromov]
A cubical complex is CAT(0) if and only if it is simply connected and the link of any vertex is a flag simplicial complex.
\end{thm}

We start by giving a global version of Gromov's theorem: we propose a combinatorial description of CAT(0) cube complexes, which is very similar but more compact than the one given by Sageev \cite{Sageev} and Roller \cite{Roller}. We first offer an informal description of our construction, and then prove its correctness by showing that it is equivalent to theirs.

\begin{figure}[h]
\begin{center}
\includegraphics[width=2.5in]{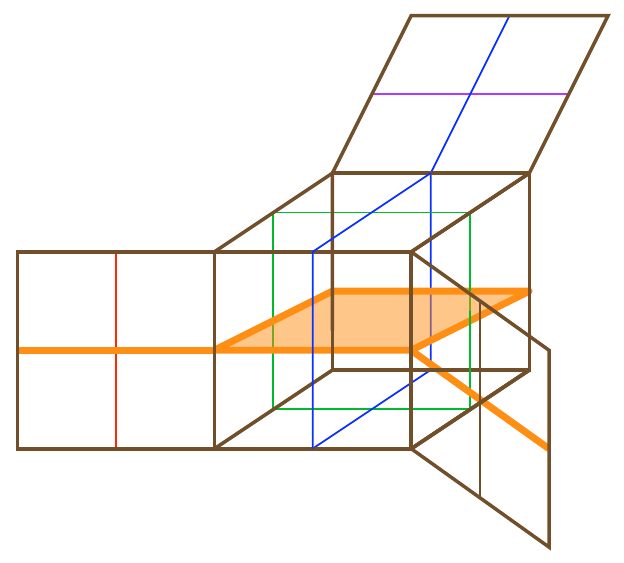} 
\caption{A cube complex $X$ and its hyperplanes.}
\label{fig:hyperplanes}
\end{center}
\end{figure}

We first describe the hyperplanes in a CAT(0) cube complex $X$, following \cite{Sageev}.
Given an $n$-cube $Q$ in $X$ and an edge $e$ in $Q$, let $Q(e)$ denote the $(n-1)$-dimensional subcube obtained by intersecting $Q$ with the hyperplane orthogonal to $e$ which passes through the midpoint of $e$. This defines an equivalence relation on the edges of each cube $Q$ by setting $e \sim e'$ if $Q(e)=Q(e')$. We extend this transitively to an equivalence relation on all the edges of $X$. Each equivalence class $\textsf{e}$ of edges defines a \emph{hyperplane} $H(\textsf{e}) = \cup Q(e)$ where the union is over all $e \in \textsf{e}$ and all cubes $Q$ containing $e$. Let $\calH(X)$ be the set of hyperplanes of $X$.
Figure \ref{fig:hyperplanes} shows a cube complex with six hyperplanes; one of the hyperplanes is shaded.

Now fix a vertex $v$ of $X$. We call $(X, v)$ a \emph{rooted} CAT(0) cubical complex.
To each hyperplane $H$ one can associate a unique vertex $h$ which is closest to $v$ and on the opposite side of $H$. These elements are highlighted in Figure \ref{fig:bijection} and numbered $1,2,3,4,5,6$. We regard the vertices of $X$ as a poset $L(X,v)$ with minimum element $v$, by decreeing that $u_1 < u_2$ if there is an edge geodesic (a shortest path along the edges of $X$) from vertex $v$ to vertex $u_2$ which passes through $u_1$. Observe that the marked elements are the elements which cover exactly one element in the poset. In fact, from the proof of the upcoming Theorem \ref{th:poset}, it follows  that $L(X,v)$ is a meet-semilattice, so we have a bijection between the hyperplanes of $X$ and the join irreducible elements of $L(X,v)$.

\begin{figure}[h]
\begin{center}
\includegraphics[width=3.5in]{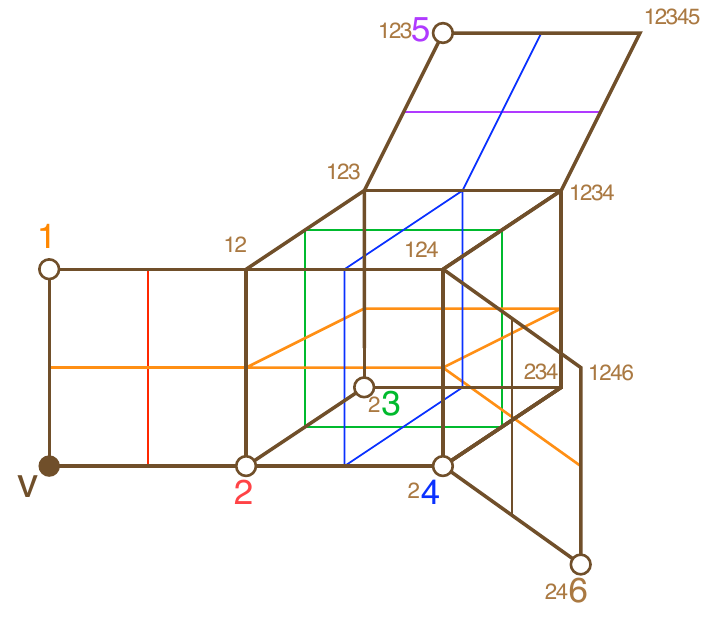} 
\caption{A cube complex $X$ and its join irreducible elements.}
\label{fig:bijection}
\end{center}
\end{figure}

The situation is now reminiscent of Birkhoff's theorem, which gives a bijection between distributive lattices and posets, as we now recall. An \emph{order ideal} or \emph{downset} $I$ of $P$ is a subset of $P$ such that $a \leq b$ and $b \in I$ imply $a \in I$. For any poset $P$, the set $J(P)$ of order ideals of $P$, partially ordered by inclusion, is a distributive lattice. Conversely, if $L$ is a distributive lattice and $P$ is the set of join-irreducible elements, ordered as in $L$, then $L \cong J(P)$. One suspects that this analogy is particularly relevant in our situation since CAT(0) cube complexes ``look like" distributive lattices.

We wish to imitate Birkhoff's theorem, and ask whether one can determine the cube complex $X$ from the join-irreducible elements of the poset $L(X,v)$, or equivalently from its hyperplanes $\calH$. There is a natural way of labelling each vertex $x$ of $X$: its index indicates which join-irreducible elements are less than or equal to $x$ in the poset $L(X,v)$. Equivalently, its index indicates which hyperplanes separate $x$ from $v$. This is illustrated in Figure \ref{fig:bijection}.

The label of each vertex is indeed an order ideal in the subposet $P(X,v)$ of join-irreducible elements of $L(X,v)$. However, not every order ideal arises in this way; so the question becomes: Which order ideals of $P(X,v)$ appear as vertex labels in $X$? The key observation is that certain pairs of hyperplanes cannot separate the same vertex $x$ from $v$. Thus, for example, no vertex label in Figure \ref{fig:bijection} simultaneously contains the pair of indices $3$ and $6$, because it is impossible to cross both hyperplanes 3 and 6 starting from $v$. We keep track of those pairs.

The following definition makes this precise. 
Recall that a poset $P$ is \emph{locally finite} if every interval $[i, j] = \{k \in P \, : \, i \leq j \leq k\}$ is finite, and it has \emph{finite width} if every antichain (set of pairwise incomparable elements) is finite.

\begin{defn}
A \emph{poset with inconsistent pairs} is a locally finite poset $P$ of finite width, together with a collection of \emph{inconsistent pairs} $\{p,q\}$, such that:
\begin{enumerate}
\item
If $p$ and $q$ are inconsistent, then there is no $r$ such that $r \geq p$ and $r \geq q$.
\item
If $p$ and $q$ are inconsistent and $p' \geq p, q' \geq q$, then $p'$ and $q'$ are inconsistent.
\end{enumerate}
\end{defn}
In particular, notice that any two inconsistent elements must be incomparable.

The \emph{Hasse diagram} of a poset with inconsistent pairs is obtained by drawing the poset, and connecting each minimal inconsistent pair with a dotted line. An inconsistent pair $\{p,q\}$ is \emph{minimal} if there is no other inconsistent pair $\{p',q'\}$ with $p' \leq p$ and $q' \leq q$.

\begin{figure}[h]
\begin{center}
\includegraphics[width=1.5in]{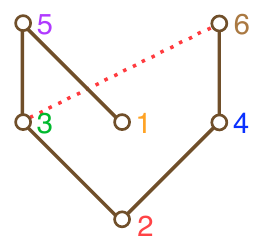} 
\caption{A poset with inconsistent pairs.}
\label{fig:poset}
\end{center}
\end{figure}

The hyperplanes of a rooted CAT(0) cube complex $(X,v)$ then form a poset with inconsistent pairs $P(X, v)$. For hyperplanes $i$ and $j$, we have that $i < j$ if, starting from $v$, one must cross hyperplane $i$ before crossing hyperplane $j$; and $i$ and $j$ are inconsistent if it is impossible to cross them both starting from $v$. This poset is shown in Figure \ref{fig:poset} for the rooted cube complex of Figure \ref{fig:bijection}.

\medskip

An \emph{antichain} $A$ of $P$ is a subset containing no two comparable elements. 
Order ideals or antichains which contain no inconsistent pair will be called \emph{consistent}; they will be particularly important to us. 

There is a bijection between consistent order ideals and consistent antichains: The maximal elements of a consistent order ideal $I$ form a consistent antichain $A=:I_{max}$, and $I$ can be recovered from $A$ as $I =  P_{\leq A} = \{p \in P \, | \, p \leq a \textrm{ for some } a \in A\}$.

\begin{defn}
If $P$ is a poset with inconsistent pairs, we construct the \emph{cube complex of $P$}, which we denote $X_P$. The vertices of $X_P$ are identified with the consistent order ideals of $P$.  There will be a cube $C(I,M)$ for each pair $(I, M)$ of a consistent order ideal $I$ and a subset $M \subseteq I_{max}$, where $I_{max}$ is the set of maximal elements of $I$. This cube has dimension $|M|$, and its vertices are obtained by removing from $I$ the $2^{|M|}$ possible subsets of $M$. These cubes are naturally glued along their faces according to their labels.
\end{defn}

For example,  if $P$ is the poset of Figure \ref{fig:poset} then $X_P$ is the complex of Figure \ref{fig:bijection}.

The following lemma is straightforward.

\begin{lemma}
Let $J$ be a consistent order ideal and let $N \subseteq J_{max}$. The faces of the cube $C(J,N)$ in the cubical complex $X_P$ are the $3^{|N|}$  cubes $C(J-N_1, N-N_1-N_2)$, where $N_1$ and $N_2$ are disjoint subsets of $N$. The maximal cubes in $X_P$ correspond to the maximal consistent antichains $A$ of $P$; they are of the form $C(P_{\leq A}, A)$.
\end{lemma}

For example, $\{1,3,4\}$ is a maximal consistent antichain in the poset of Figure \ref{fig:poset}, which corresponds to the maximal 3-cube $C(\{1,2,3,4\},\{1,3,4\})$ in Figure \ref{fig:bijection}.

Now we are ready for the main theorem of this section.

\begin{thm}\label{th:poset} (Combinatorial description of CAT(0) cubical complexes.)
There is a bijection between posets with inconsistent pairs and rooted CAT(0) cube complexes, given by the map $P \mapsto X_P$. 
\end{thm}

While it is possible to prove Theorem \ref{th:poset} directly, it will be easier to recall Sageev and Roller's description of CAT(0) cube complexes \cite{Roller, Sageev}, and prove that ours is equivalent to theirs.

A \emph{halfspace system} \cite{Sageev} or \emph{pocset} \cite{Roller} is a triple $H=(H, \leq, *)$ consisting of a set $H$ of \emph{halfspaces}, a locally finite poset $\leq$ on $H$ of finite width, and an order-reversing involution $*$, denoted $h \rightarrow h^*$, such that two halfspaces $h,k \in H$ (coming from different hyperplanes) satisfy at most one of the four inequalities:
\[
h \leq k, h \leq k^*, h^* \leq k, h^* \leq k^*.
\]
In particular $h$ and $h^*$ must be incomparable for any $h \in H$.

\begin{figure}[h]      
\begin{center}
\includegraphics[width=3in]{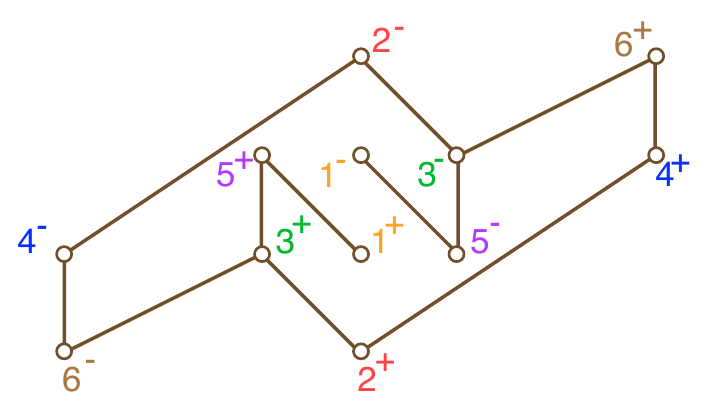} 
\caption{A halfspace system $H$.}
\label{fig:signedposet}
\end{center}
\end{figure}

Each element of $H$ is called a \emph{halfspace}, and two halfspaces are called \emph{nested} if one of the inequalities above holds, or \emph{transversal} if none of them hold. A pair $\{h, h^*\}$ is called a \emph{hyperplane}. Let $H^0$ be the set of hyperplanes. 

It will be useful for us to choose an arbitrary orientation for each hyperplane, and label the elements of the pair $h^+$ and $h^-$. We will do this from now on, and denote the hyperplane $h=\{h^+, h^-\}$.

Once again, this construction is motivated by the geometry of a CAT(0) cube complex $X$, as illustrated in Figure \ref{fig:signed.hyps}. Each combinatorial hyperplane $h=\{h^+, h^-\}$ represents a geometric hyperplane of $X$, which divides $X$ into two halfspaces $h^+$ and $h^-$, with an arbitrary choice of sign. The poset on $H$ represents the poset of containment of these geometric halfspaces. For example, Figure \ref{fig:signedposet} is the poset of the halfspaces of the cube complex of Figure \ref{fig:signed.hyps}.

\begin{figure}[h]
\begin{center}
\includegraphics[width=4in]{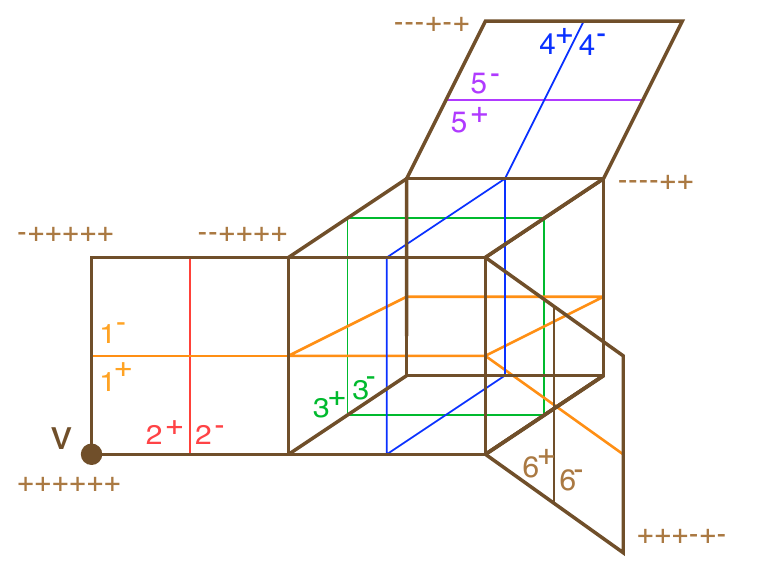} 
\caption{The CAT(0) cube complex $X_H$ of the
halfspace system $H$ of Figure \ref{fig:signedposet}
;  for brevity, vertex  $1^-2^-3^-4^+5^-6^+$ is denoted $---+-+$.}
\label{fig:signed.hyps}
\end{center}
\end{figure}

To a halfspace system $H$, Roller and Sageev associate a cube complex $X_H$. Its vertices $v = \{h^{v(h)}:h \in H^0\}$ correspond to the choices of a sign for each element of $H^0$ such that $h^s \not\leq k^{-t}$ for all $h^s$ and $k^t$ in $v$. Recall that an \emph{order filter} or \emph{upset} $I$ of $P$ is a subset of $P$ such that $a \geq b$ and $b \in I$ imply $a \in I$. Then each vertex $u$ of $X_H$ corresponds to an order filter of the poset $H$ which contains exactly one halfspace from each hyperplane. Geometrically, in the cube complex, these are the halfspaces containing $u$.

To describe the cubes that a vertex $u$ is in, regard $u$ as an order filter $F$ of the poset $H$. Choose any $d$ minimal elements of $F$.
These $d$ elements are pairwise transversal halfspaces, and one can
change their signs in any way to obtain the $2^d$ vertices of a
$d$-cube containing $u$.

Maximal cubes of $X_H$ correspond to maximal sets $A$ of pairwise transverse hyperplanes of $H$. The order filter $H_{\geq A} = \{h^s \in H \, : \, h^s \geq a^+ \textrm{ or } h^s \geq a^- \textrm{ for some } a\in A\}$ contains both halfspaces of each hyperplane in $A$ and exactly one halfspace of each hyperplane not in $A$. The vertices of the cube correspond to the $2^{|A|}$ ways to choose a halfspace for each hyperplane in $A$ and remove them from $H_{\geq A}$.

For example, $A=\{1, 3, 4\}$ is a maximal set of pairwise transverse hyperplanes in the halfspace system of Figure \ref{fig:signedposet}. It corresponds to the three-dimensional cube in Figure \ref{fig:signed.hyps}, with fixed vertex labels $2^-, 5^+$ and $6^+$ (which are above $A$ in $H$), and arbitrary choices of signs for hyperplanes $1, 3,$ and $4$.

\begin{thm}[Sageev \cite{Sageev}, Roller \cite{Roller}] \label{th:sageev}
If $H$ is a halfspace system then $X_H$ is a CAT(0) cube complex. Conversely, every CAT(0) cube complex arises in this way from a halfspace system.
\end{thm}

Now we are ready to prove the main result of this section.

\begin{proof}[Proof of Theorem \ref{th:poset}.] 
Say a halfspace system $H$ is \emph{acyclic} if  it has no order relations of the form $a^+ < b^-$. This is equivalent to saying that the all-positive set $\{h^+: h \in H^0\}$ is a vertex of $X_H$, which explains the terminology, borrowed from oriented matroid theory.

We proceed in three steps, as follows:

\noindent \emph{a}. Every CAT(0) cube complex can be obtained from an acyclic halfspace system. 

\noindent \emph{b}.
Posets with inconsistent pairs are in bijection with acyclic halfspace systems. 

\noindent \emph{c}.
The cube complex that Theorem \ref{th:poset} associates to a poset with inconsistent pairs is the same one that Theorem \ref{th:sageev} assigns to its corresponding acyclic halfspace system.

\medskip

\emph{a}.
Let $X=X_H$ be any CAT(0) cube complex, which comes from an arbitrary halfspace system $H$. Fix a vertex $v$ of $X_H$, and reverse the labels of $h^+$ and $h^-$ for each hyperplane $h$ such that $h^+\in v$.  The resulting halfspace system $H'$ is acyclic, and it is clear that $X_H$ and $X_{H'}$ are equal up to the aforementioned relabeling.

\medskip

\emph{b}.
Let $H$ be an acyclic halfspace system, and let $v_0=\{h^+ : h \in H^0\}$ be the all-positive vertex of $X_H$.  Consider the poset on $H^0$ which one obtains by restricting the poset $(H, \leq)$ to $H^+$, and decree the pair $\{p,q\}$ to be inconsistent whenever $p^- < q^+$ in $H$. To see that $H^0$ is indeed a poset with inconsistent pairs, we need to check two things:

1. If $p$ and $q$ are inconsistent, then $p$ and $q$ have no common upper bound in $H^0$: If $r$ was such an upper bound, we would have $p^+ \leq r^+$ and $q^+ \leq r^+$ in $H$, which together with $p^- < q^+$ would give $p^-,p^+ \leq r^+$ contradicting the definition of a halfspace system. 

2. If $p$ and $q$ are inconsistent, any $p', q' \in H^0$ with $p' \geq p$ and $q' \geq q$ are inconsistent: We have $(p')^- \leq p^-$ by the order reversing involution, $p^- \leq q^+$ by the inconsistency of $p$ and $q$, and $q^+ \leq (q')^+$ by assumption. Therefore $(p')^- \leq (q')^+$ as desired.

To recover $H$ from $H^0$, make a positive and a negative copy of each element of $H^0$, and introduce the following order relations on each 4-tuple $\{a^+, a^-, b^+, b^-\}$: If $a<b$ in $H^0$, let $a^+<b^+$ and $b^-<a^-$ in $H$. If $a$ and $b$ are inconsistent in $H^0$, let $a^- < b^+$ and $b^- < a^+$ in $H$. Finally, if $a$ and $b$ are incomparable but consistent in $H^0$, let them be transverse in $H$, so $\{a^+, a^-, b^+, b^-\}$ are pairwise incomparable in $H$.

For example, this bijection maps the poset with inconsistent pairs in Figure \ref{fig:poset} to the acyclic halfspace system of Figure \ref{fig:signedposet}.
 
\medskip

\emph{c}.
It remains to show that the complex $X_H$ of Roller and Sageev is isomorphic to the complex $X_{H^0}$ that we construct. First we establish a bijection between the vertices of these complexes. To each vertex $w = \{h^{w(h)}:h \in H\}$ of $X_H$ we associate the set $S(w) \subseteq H^0$ of hyperplanes $h$ such that $h^- \in v$. 
 To show that $S(w)$ is an order ideal of $H^0$, assume that $a \leq b$ in $H^0$ and $b \in S(w)$. This means that $b^- \in w$ and $a^+ \leq b^+$ in $H$, which implies $b^- \leq a^-$ in $H$. Since $H$ is an order filter, we have $a^- \in w$; that is, $a \in S(w)$.
Also, if $S(w)$ contained an inconsistent pair $\{a,b\}$, that would mean that $a^-< b^+$ for $a^-, b^- \in w$, a contradiction to the definition of the vertices of $X_H$. Therefore $S(w)$ is a consistent order ideal, \emph{i.e.}, a vertex of $X_{H^0}$. Conversely, given a consistent order ideal $S(w)$ it is clear how to recover $w$; and keeping in mind the acyclicity of $H$, the argument can be reversed to show that this $w$ is indeed a vertex of $X_H$. 

Finally, $X_H \cong X_{H^0}$ follows from the description of the maximal cubes in these two complexes.
\end{proof}

\begin{remark}
Our combinatorial description of CAT(0) cubical complexes is different from Sageev's in that it pays special attention to one particular vertex of the complex, and it breaks the symmetry between the positive and negative sides of a hyperplane. For the purposes of this paper, this feature of our description is advantageous.

In general, which description is more useful depends on the particular application. Ours is particularly helpful when there is a ``special" vertex, or when there is no harm in choosing one. In the Billera-Holmes-Vogtmann tree space, the origin might play that role. In a cubical complex acted on by a group, the ``special" vertex might represent the identity. In a reconfigurable system, we might choose a ``home" state. 
\end{remark}

Let $P$ be a poset with inconsistent pairs.  There are many ways to embed the associated cubical complex $X_P$ in a finite dimensional real vector space.  We will often make recourse to the following embedding, which we call the \emph{standard embedding}.

\begin{eqnarray*}
X_P & = &  \left\{(x_1, ..., x_s) \in [0,1]^{|P|} : i \prec j \textrm { and } x_i <1 \implies x_j = 0 , \right.\\
&  & \left. \, \, \, \textrm{ and if }(i,j)\textrm{ are inconsistent, then } x_i x_j = 0\right\}
\end{eqnarray*}

It is useful to literally think of the points in $X_P$ as assignments of a real number between 0 and 1 to each element of $P$. To move around $X_P$ starting at $v$, we start by assigning $0$s to all vertices, and then increase these numbers following two rules: no two incomparable elements are allowed to have non-zero numbers, and to increase the number in a certain position, one must have first increased all numbers in lower positions to $1$.

Note that the standard embedding has the property that any refinement of $P$, obtained by adding relations or inconsistent pairs, is a subcomplex under its standard embedding.


\section{Two  applications: embeddability and realizability}\label{sec:applications}

We now present two applications of the combinatorial description of CAT(0) cube complexes of  Section \ref{sec:combo}.

\subsection{Every $n$-dimensional interval of a CAT(0) cube complex embeds into $\zz^n$}\label{sec:interval}

In the section we define intervals in a CAT(0) cube complex, and use Theorem \ref{th:poset} to answer a question of Niblo, Sageev, and Wise \cite{NSW}.

\begin{defn}\label{def:interval}
Given two vertices $v,w$ of a CAT(0) cube complex $X$, let $[v,w]$ be the set of vertices which lie on at least one edge geodesic between $v$ and $w$.  Alternatively, these are the vertices which lie in every halfspace that contains $v$ and $w$. Let the \emph{interval} $X[v,w]$ in $X$ be the subcomplex of $X$ consisting of all cubes whose vertices are in $[v,w]$.

If $C$ and $D$ are two cubes in $X$, let $[C,D]$ be the set of vertices situated on at least one edge geodesic between a vertex in $C$ and a vertex in $D$.  Define the \emph{interval} $X[C,D]$ to be the subcomplex of $X$ consisting of all cubes whose vertices are in $[C,D]$.
 \end{defn}

\begin{lemma}\label{lem:vertices.are.order-ideals}
Consider an interval $[v,w]$ in a cube complex $X$. Root $X$ at $v$, and write $X = X_P$ for the corresponding poset with inconsistent pairs $P$.  Then the vertex $w$ corresponds to a consistent order ideal $Q$ of $P$, and the vertices of the interval $X[v,w]$ correspond to the order ideals of $Q$. In particular, $X[v,w] \cong X_Q$.
\end{lemma}

\begin{proof}
The vertex $w$ corresponds to a consistent order ideal $Q$ of $P$, and the edge geodesics from $v$ to $w$ correspond to the ways to build up the order ideal $Q$ by adding one element at a time; \emph{i.e.}, to the sequences of order ideals $\emptyset = I_1 \subset I_2 \subset \cdots \subset I_k = Q$ with $|I_i-I_{i-1}| = 1$ for all $i$. Since $Q$ contains no inconsistent pairs, all of its order ideals are consistent, and they are in bijection with the vertices of the interval $X[v,w]$. It then follows that $X[v,w] \cong X_Q$ as well.
\end{proof}

 \begin{remark}
The two notions of interval in Definition \ref{def:interval} coincide. For any two vertices $v$ and $w$ in  $X$, let $C$ and $D$ be the largest cubes in $X[v,w]$ containing $v$ and $w$, respectively. Then one easily checks that $X[v,w] = X[C,D]$.

Conversely, let $C$ and $D$ be cubes in $X$. We can choose vertices $v$ of $C$ and $w$ of $D$ which are farthest from each other, so that any hyperplane intersecting either $C$ or $D$ separates $v$ and $w$. Again, one easily checks that $X[C,D] = X[v,w]$.
 \end{remark}

\begin{defn}
Let the \emph{dimension} of an interval in a CAT(0) cube complex be the dimension of the largest cube in that interval.
 \end{defn}

In \cite{NSW}, Niblo, Sageev, and Wise asked for a proof of the following result, which had been given a flawed proof in the literature. Theorem \ref{th:poset} gives us a simple proof. (Brodzki, Campbell, Guentner, Niblo, and Wright \cite{BCGNW} and Chepoi and Maftuleac \cite{ChepoiMaftuleac10} independently found proofs similar to ours.)

\begin{thm}\label{th:embedding}
Any interval of dimension $n$ in a CAT(0) cube complex embeds as a subcomplex in the integer lattice cubing of $\rr^n$.
\end{thm}

\begin{proof}
Let the interval be $X[v,w]$. Point the cube complex $X$ at $v$ and let $X=X_P$ be the cube complex of the poset with inconsistent pairs $P$.  Let $Q$ be the consistent order ideal corresponding to the vertex $w$. 

Partially order $[v,w]$ by letting $w_1 \leq w_2$ if some edge geodesic from $v$ to $w_2$ goes through $w_1$. Combinatorially, this is equivalent to saying that the corresponding order ideals $W_1$ and $W_2$ of $Q$ satisfy $W_1 \subseteq W_2$. As a poset, $[v,w]$ is isomorphic to $J(Q)$, the poset of order ideals of $Q$. By Birkhoff's theorem, $J(Q)$ is a distributive lattice.

The \emph{order dimension} of a poset is the least positive integer $n$ for which the poset can be embedded as a subposet of $\zz^n$ with the componentwise partial order. Dilworth \cite{Dilworth} proved that the order dimension of a distributive lattice $J(Q)$ equals $q:=\textsf{width}(Q)$, the size of the largest antichain of $Q$.

The following embedding is due to Reading \cite{Reading}: By Dilworth's theorem, $Q$ can be decomposed as the disjoint union of $q$ chains $C_1, \ldots, C_q$. To embed $J(Q)$ into $\zz^q$, map an order ideal $R \subseteq Q$ to the point $f(R) = (|R \cap C_1|, \ldots, |R \cap C_q|)$. Notice that if $R_2$ covers $R_1$ in $J(Q)$ then the edge between $f(R_1)$ and $f(R_2)$ is a unit edge in the direction $i$ such that $C_i$ contains the element $R_2-R_1$.

This means that the 1-skeleton of $X$ embeds as a subcomplex of the 1-skeleton of the integer lattice cubing of $\rr^n$. 
One can then embed the whole cube complex by filling in the appropriate cubes: whenever we need to fill in a cube in $X$, the corresponding edges in $\rr^n$ form a unit cube which is part of the integer lattice cubing.
\end{proof}

\subsection{Every finite CAT(0) cube complexes comes from a reconfigurable system.}\label{sec:realizable}

An important source of examples of CAT(0) cube complexes is the theory of reconfigurable systems. In this theory, one starts with a graph $G$ and a set $A$ of possible vertex labels. A \emph{state} is a labeling of each vertex of $G$ with an element of $A$. There are certain \emph{moves} or \emph{generators} that one can perform. Each move $\phi$ has a prescribed support $SUP(\phi) \subseteq V(G)$ which determines the legality of a move, and a trace $TR(\phi) \subseteq SUP(\phi)$ where the move takes effect. The move also has a labeling $u: SUP(\phi) - TR(\phi) \to A$ and two labelings $u_1, u_2: TR(\phi) \to A$. The move $\phi$ is \emph{admissible} at a state $v$ if $v$  agrees with $u$ on $SUP(\phi) - TR(\phi)$ and it agrees with one of $u_1, u_2$ on $TR(\phi)$. In that case, the effect of the move is to switch the local labeling of $TR(\phi)$ from $u_1$ to $u_2$ or vice versa. (In particular, if $\phi$ is admissible at $v$, then it is also admissible at $\phi(v)$, and $\phi(\phi(v)) = v$.) A \emph{reconfigurable system} is a collection of moves and a collection of states which is closed under those moves.

Innumerable systems which change according to local rules can be naturally modeled as reconfigurable systems. Examples include the motion planning of a robot, the prevention of collision among several robots, the motion of a set of particles on a graph, and the folding of proteins, among many others; see \cite{GhristPeterson07}.

In many reconfigurable systems, the \emph{parameter space} of all possible positions of the system naturally takes the shape of a CAT(0) cubical complex $X$. 
Say that two moves $\phi_1$ and $\phi_2$ are said to \emph{commute} if they are ``physically independent" and can be applied simultaneously; that is, if $TR(\phi_0) \cap SUP(\phi_1) = TR(\phi_1) \cap SUP(\phi_0) = \emptyset$.
Then the vertices of $X$ correspond to the states, the edges correspond to the moves connecting two states, and the cubes correspond to sets of pairwise commuting moves, which can be applied simultaneously to a state. Ghrist and Peterson \cite{GhristPeterson07} showed that the result is a cubical complex of \emph{local} non-positive curvature. This means that the link of every vertex is a flag simplicial complex, but the cubical complex is not necessarily simply connected.  

They also gave an indirect proof of a stronger converse:  that every CAT(0) cubical complex is the state complex of a reconfigurable system. With Theorem \ref{th:poset} in hand, we can now give a simple constructive proof:

\begin{thm}[Ghrist-Peterson] Any finite CAT(0) cubical complex is realizable as the state complex of a reconfigurable system.
\end{thm}

\begin{proof}
Root the given CAT(0) cube complex $X$ at a vertex $v$ and let it correspond to a (finite) poset with inconsistent pairs $P$. We construct a reconfigurable system which represents a virus trying to take over $P$, starting at the bottom and working its way up the poset. The comparability relations $i<j$ help transmit the virus up the poset, while the inconsistency relations prevent it from spreading.

The underlying graph of the system is the Hasse diagram of $P$ and the set of labels is $\{0,1\}$. 
For each element $p \in P$ there is a move $\phi_p$ which infects it, changing
its label from a $0$ to a $1$ or vice versa. 
To apply the move $\phi_p$, it is required that 
the elements covered by $p$ are labeled 1, and that the minimal elements among those inconsistent with $p$ are labeled 0.

 The collection of states is the collection of consistent order ideals of $P$, encoded as 0-1 labelings of $P$, where the $1$ labels denote the elements of the ideal. This is clearly closed under the set of moves. It is also clear that the state complex of this reconfigurable system is isomorphic to $X$.
\end{proof}

From this point of view, 
the space $X$ serves as a continuous model for a discrete reconfigurable system, with points of $X$ representing positions of the continuous model. 
Finding geodesics in $X$ is equivalent to finding the optimal way to get the system from one position to another one under this particular metric. 
Abrams and Ghrist \cite{AbramsGhrist04} consider the analogous problem under a different metric, which assumes that all moves take the same amount of time, and physically independent moves can be done simultaneously at no additional cost. Our metric even allows independent moves to be done simultaneously \emph{at different speeds}, but now there is a `Pythagorean' penalty for doing so. For instance, if each one of  two independent moves demand 1 unit of time, energy, or cost, then we can do them both simultaneously in $\sqrt{2}$ units. As should be expected, our geodesic will be different from theirs in general.

\section{Combinatorics of Geodesics}
\label{sec:combin.geodesics}

In this section, we study the combinatorics of finding the geodesic from $x$ to $y$ in a CAT(0) cube complex.  
The first step will be to root the complex at a vertex $v$ which is intuitively ``near $x$" and ``in the opposite direction from $y$", as follows. Let $V$ and $W$ be the minimal cubes containing $x$ and $y$ respectively. We choose vertices $v$ and $w$ of $V$ and $W$, which we describe by their position with respect to the hyperplanes. If a hyperplane $H$ does not intersect $V$, the position of $v$ with respect to $H$ is determined; similarly for $w$ and $W$. If $H$ intersects $V$ and does not intersect $W$, choose $v$ to be on the opposite side of $H$ from $W$, and vice versa. Finally, if $H$ intersects both $V$ and $W$, choose $v$ and $w$ so that these four points are in relative position $v,x,y,w$ with respect to the direction orthogonal to $H$. In other words, choose them so that, when we root the complex at $v$, we have $0 = v_H \leq x_H \leq y_H \leq w_H = 1$ in the standard embedding of $X$. (If $x_H = y_H$ then we have some freedom in our choice of $v$.)

Now our rooted complex has an associated poset with inconsistent pairs $P$.  
Let $Q$ be the consistent order ideal corresponding to $w$.  Then $Q$ contains no inconsistent pairs and $X[v,w] = X_Q$.  The only non-zero coordinates in $X_Q$ are those corresponding to elements of $Q$. The following lemma (which also appears in  \cite{ChepoiMaftuleac10} 
and is used there in the 2-dimensional case) allows us to use the tools from the previous sections in the geodesic problem.

\begin{lemma} \label{lem:geoInterval}
The geodesic from $x$ to $y$ is contained in $X[v,w]$. 
\end{lemma}

\begin{proof}
Suppose not.  Then at some point $q$ the geodesic exits $X_Q$, before re-entering at point $r$, since $y$ is in $X_Q$.  Let $\gamma_{qr}$ be the segment of the geodesic between $q$ and $r$.  Project $\gamma_{qr}$ onto $X_Q$ by sending each point $z = (z_1, ..., z_s)$ to $\tilde{z} = (\tilde{z}_1, ... \tilde{z}_s)$ where 

\[
\tilde{z}_i=\left\{\begin{array}{ll}
z_i,
&i \in Q\\
0,
&i \notin Q \\
\end{array}\right. ,
\]
which is easily checked to be in $X_Q$. Since any two points in $\gamma_{qr}$ which lie in the same cube $C(I,M)$ in $X_P$ get mapped to points that lie the same cube $C(\tilde{I}, \tilde{M})$ in $X_Q$, the distance of the projected points is less than or equal to their distance in $\gamma_{qr}$.
This implies that the projection $\tilde{\gamma}_{qr}$ of $\gamma_{qr}$ onto $X_Q$ is a path of length less than or equal to the length of $\gamma_{qr}$.  Since the geodesic in a CAT(0) complex is unique, $\gamma_{qr}$ could not be the geodesic.  
\end{proof}

Thus the geodesic must lie in some connected sequence of cubes in $X_Q$ that begins with the cube $V$ containing $x$ and ends with the cube $W$ containing $y$.  We will now introduce a certain kind of cube sequence, called a \emph{valid cube sequence}.  We will then show that valid cube sequences have some nice properties, and that the geodesic is contained in a valid cube sequence.

\begin{defn}
Let $V$ and $W$ be the minimal cubes in $X_Q$ containing $x$ and $y$, respectively.  A \emph{valid cube sequence} is a sequence  $V = C(I_1, M_1), C(I_2, M_2), \ldots , C(I_k, M_k) = W$ of cubes in $X_Q$ such that
\begin{enumerate}
\item $I_1\subset I_2 \subset \cdots \subset I_k = Q$;
\item  $I_1 = M_1$, $ I_i \bs I_{i-1} \subseteq M_i$ for $1< i \leq k$;
\item $M_i$ is a maximal anti-chain in $Q$ for all $1 \leq i \leq k$.
\end{enumerate}
\end{defn}

By condition (3), the cube sequence is determined uniquely by the order ideals $I_1, \ldots, I_k$. We can therefore think of it as a \emph{partial linear extension}, \emph{i.e.}, an order-preserving function $f:P \to [k]$, where $f(p)$ is the smallest $i$ such that $p \in I_i$; this is illustrated in Figure \ref{fig:partialExtValidCubeSeq}.

\begin{ex}\label{ex:valid.cube.seq}
Suppose we want to find the geodesic between points $x$ and $y$, and have identified the interval $X[v,w] \cong X_Q$ shown in Figure~\ref{fig:validCubeSeq}.  The corresponding poset $Q$ is given in Figure~\ref{fig:posetValidCubeSeq}.  The shaded cubes form a valid cube sequence  $C(I_1, M_1)$, $C(I_2, M_2)$, $C(I_3, M_3)$, $C(I_4, M_4)$, $C(I_5, M_5)$, where $I_1 = \{1,5\}$, $I_2 = \{1,2,5\}$, $I_3 = \{1, 2, 5, 6\}$, $I_4 = \{1, 2, 3, 5,6, 7\}$, $I_5 = \{1, 2, 3, 4, 5, 6, 7\}$, and $M_1 = \{1,5\}, M_2 = \{2, 5\}, M_3 = \{2,6\}, M_4 = \{3, 7\}, M_5 = \{4,7\}$.  The shortest path from $x$ to $y$ contained in this valid cube sequence is shown by the dashed line.  Figure~\ref{fig:partialExtValidCubeSeq} gives the partial linear extension corresponding to this valid cube sequence.
\end{ex}

\begin{figure}[ht]
\centering
\subfigure[]{
\includegraphics[scale = 0.5]{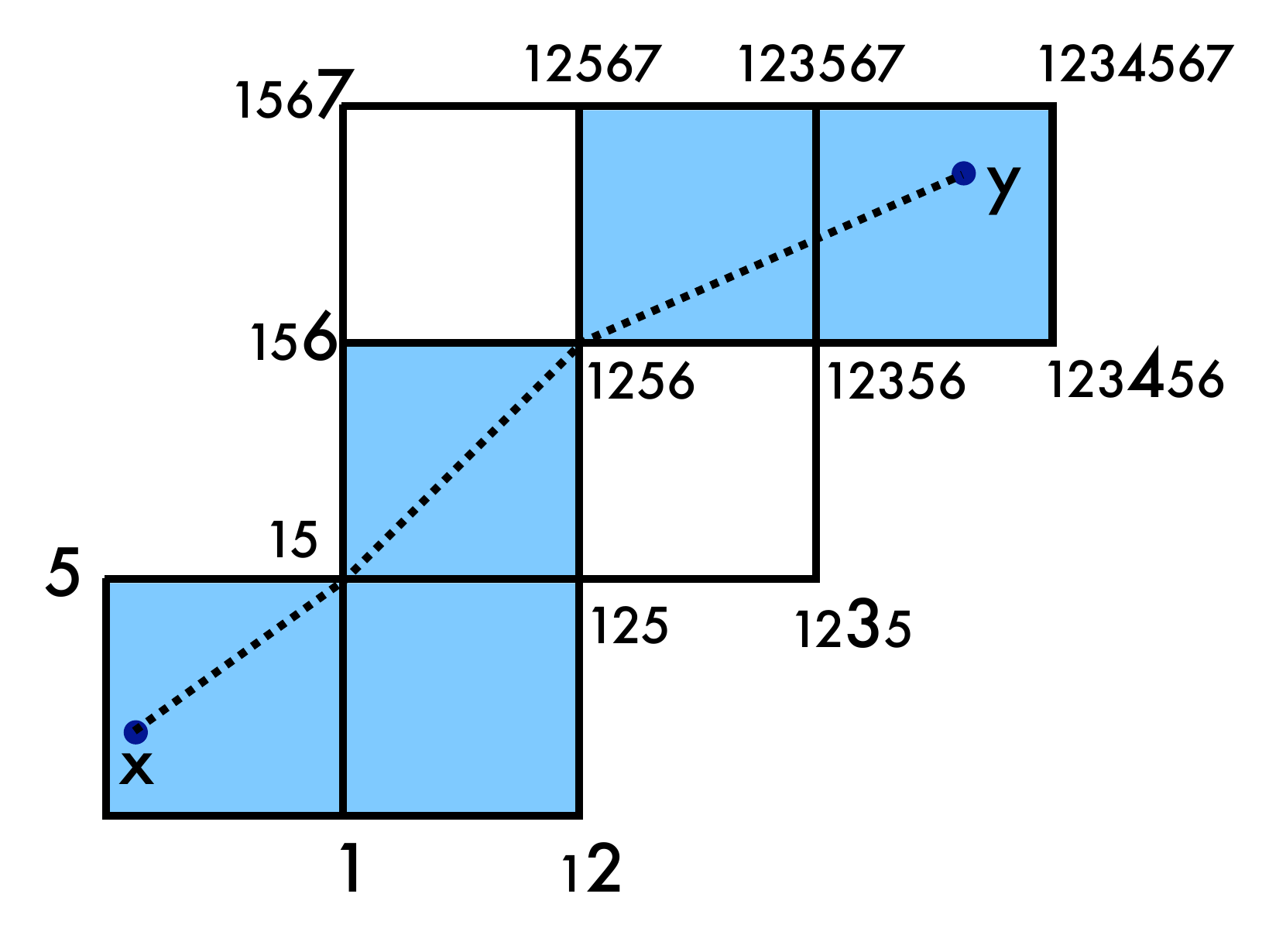}
\label{fig:validCubeSeq}
}
\subfigure[]{
\includegraphics[scale = 0.5]{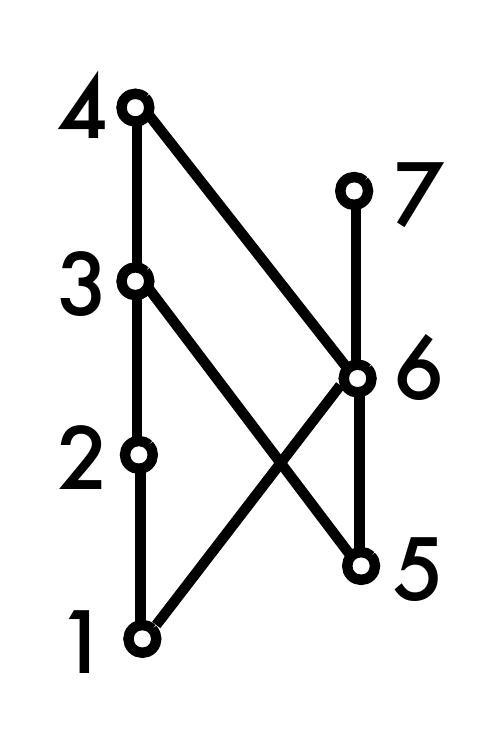}
\label{fig:posetValidCubeSeq}
}
\subfigure[]{
\includegraphics[scale = 0.5]{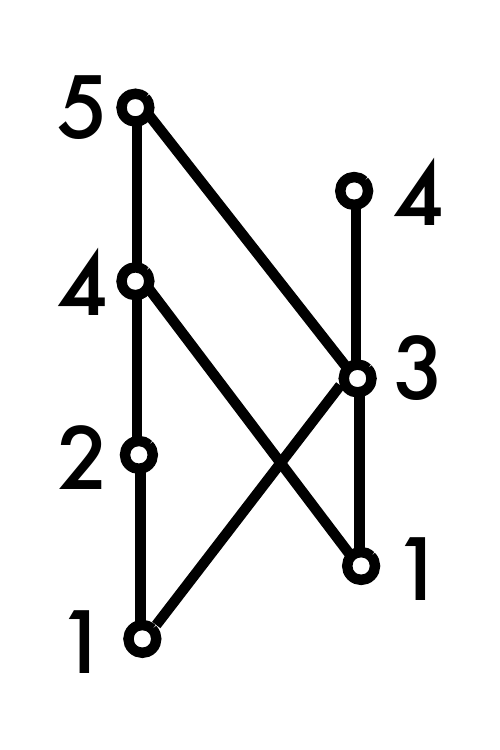}
\label{fig:partialExtValidCubeSeq}
}

\caption{Corresponding to Example~\ref{ex:valid.cube.seq}, Figure \ref{fig:validCubeSeq}
 is the interval $X[v,w]$ that contains the geodesic from $x$ to $y$.  The shaded cubes form a valid cube sequence, and the dashed line represents the shortest path from $x$ to $y$ contained in that valid cube sequence.  Figure \ref{fig:posetValidCubeSeq} is the poset of the interval $X[v,w]$, and Figure \ref{fig:partialExtValidCubeSeq} shows the partial linear extension corresponding to the shaded valid cube sequence. }
\end{figure}

\begin{ex}
An important example of a cube sequence (which is usually not valid) is the \emph{normal cube path} between the vertices $v$ and $w$, defined by Niblo and Reeve \cite{NibloReeves98} as follows.   Let $C_1$ be the minimal cube containing $x$.  Starting at $v$, travel through the cube $C_1$ to the vertex which is closest to $w$ in the edge geodesic metric.  Define $v_1$ to be this vertex, and iterate the process starting at $v_1$. In the end we get a  cube sequence $C_1, \ldots, C_k$ and a sequence of vertices $v=v_0, v_1, \ldots, v_{k-1}, v_k=w$ where $C_i \cap C_{i+1} = v_i$ and $C_i=X[v_{i-1},v_i]$. Incidentally, to construct an edge geodesic from $v$ to $w$, we just need to string together minimal edge paths from $v_{i-1}$ to $v_i$ along $C_i$ for each $i$.

This construction is easily described in terms of the corresponding poset with (no) inconsistent pairs. As above, root $X[v,w]$ at $v$ and let $Q$ be the consistent order ideal corresponding to vertex $w$, so $X[v,w]=X_Q$. Then $C_1, \ldots, C_k$ are obtained by iteratively pruning off all minimal elements of the poset $Q$.

A normal cube path is not necessarily a valid cube sequence, since some of its cubes may not be maximal.  For example, if $v$ and $w$ are opposite corners of the $n \times k$ grid with $n \neq k$, then the normal cube sequence will proceed on a diagonal from $v$ towards $w$ until it hits an edge of the grid.   
From this point on, the cubes will be edges, and hence not maximal.

However, we can easily modify a normal cube path to make it a valid cube sequence, by replacing non-maximal cubes with maximal ones containing them.  If $C_i = C(I_i, M_i)$ is the first non-maximal cube, we can replace it with $C'_i = C(I_i, M'_i)$ where $M'_i$ is the set of maximal elements of $I_i$, and iterate.  The result is clearly a valid cube sequence, which we call the \emph{extended normal cube path}.

\end{ex}

\begin{prop}\label{prop:validproperties}
A valid cube sequence $C(I_1, M_1), C(I_2, M_2), ..., C(I_k, M_k)$ has the following properties:

\begin{enumerate}
\item It contains $x$ and $y$.

\item  It is connected.

\item  The intersection $C(I_{i-1}, M_{i-1}) \cap C(I_i, M_i)$ is the face $C(I_{i-1}, M_{i-1} \cap M_i)$.

\item Each cube $C(I_i, M_i)$ is maximal.

\end{enumerate}
\end{prop}

\begin{proof}
Property 1 holds since the first and last cubes are the minimal ones in $Q$ containing $x$ and $y$, respectively.

Property 2 follows from property 3, so we now show that $C(I_{i-1}, M_{i-1}) \cap C(I_i, M_i) = C(I_{i-1}, M_{i-1} \cap M_i)$.  Note that $C(I_{i-1}, M_{i-1} \cap M_i)$ is not empty by definition.  Let $u$ be a vertex in $C(I_{i-1},M_{i-1} \cap M_i)$, with corresponding order ideal $I_{i-1} \bs M$, where $M$ is contained in $M_{i-1} \cap M_i$, and thus in $M_{i-1}$.  So $u$ is a vertex of $C(I_{i-1}, M_{i-1})$.  We can rewrite $I_{i-1} \bs M = I_i \bs ((I_i \bs I_{i-1}) \cup M)$.  Since both $I_i \bs I_{i-1}$ and $M$ are contained in $M_i$, we also have that $u$ is a vertex of $C(I_i, M_i)$.  Thus $C(I_{i-1}, M_{i-1} \cap M_i) \subseteq C(I_{i-1}, M_{i-1}) \cap C(I_i, M_i)$.  

Next let $u$ be a vertex in $C(I_{i-1}, M_{i-1}) \cap C(I_i, M_i)$.  Since $v \in C(I_{i-1}, M_{i-1})$, the order ideal corresponding to $u$ is $I_{i-1} \bs M$, for some $M \subseteq M_{i-1}$.  If $M \nsubseteq M_i$, then there would exist some element $a \in M \subseteq I_{i-1} \subseteq I_i$ that is not in $M_i$.  This would imply that $a$ is in all of the order ideals corresponding to vertices in $C(I_i, M_i)$.  But $a$ is not in the order ideal $I_{i-1} \bs M$ of $u$, which is a vertex of $C(I_i, M_i)$, a contradiction.  Thus $M \subseteq M_i$, and hence $u$ is in $C(I_{i-1}, M_{i-1} \cap M_i)$.   

Finally, property 4 follows from the condition that $M_i$ is a maximal anti-chain in $Q$.
\end{proof}

\begin{lemma}\label{lem:geoInValidCubeSeq}
The geodesic is contained in a valid cube sequence.
\end{lemma}

\begin{proof}
The set of valid cube sequences contains all connected sequences of maximal cubes containing $x$ and $y$ such that we always have larger order ideals at each step.  Clearly the geodesic must pass through a connected set of maximal cubes containing $x$ and $y$.  We must show that the order ideals corresponding to the cubes of the sequence always increase as we move along the geodesic from $x$ to $y$.  To this end it suffices to show that we never gain an element $i$ in the order ideal and then lose it again.  If this happened, the geodesic $\gamma$ would leave the plane $x_i = 0$ at some point $r$
and then return to it at some point $r'$ (in the standard embedding).  Project the path between $r$ and $r'$ onto the plane $x_i = 0$, where the projection of the point $z$ defined by: 
\[
\tilde{z}_j=\left\{\begin{array}{ll}
0,
&j \geq_Q i  \\
z_j,
& \textrm{otherwise}\\
\end{array}\right. ,\]
which is easily seen to be in $X_Q$.
This gives a path that eliminates this extra addition and subtraction of the element $i$, and is no longer than $\gamma$, contradicting that $\gamma$ is the unique geodesic.   
\end{proof}

\section{A Characterization of Geodesics}\label{sec:char}

We now characterize geodesics in CAT(0) cube complexes.  To do this, we introduce two new properties, the Zero-Tension Condition and the No Shortcut Condition, which we will show all geodesics satisfy.  We then show that if we have a path contained in a valid cube sequence that satisfies the Zero-Tension and No Shortcut conditions, then it must be the geodesic.  Algorithm \ref{alg:geodesic} for computing geodesics will follow from this characterization.

\subsection{The Zero-Tension Condition}\label{sec:zero}

In this subsection we describe the Zero-Tension condition, which allows us to check whether or not a path through a given sequence of cubes is a geodesic. 
We describe the condition for general polyhedral complexes.

Let $\calc$ be a polyhedral complex with the induced Euclidean metric on each polyhedral cell.  We imagine that each maximal cell is embedded in some specific Euclidean space.  The attaching maps that connect polyhedral cells are assumed to be injective and isometric.  We assume that there are only finitely many cells and that all polyhedra are convex.

Let $x,y$ be two points in $\calc$ that are in different cells.  Since each of the cells has a Euclidean metric,  any connected component of any geodesic between $x$ and $y$ that is contained in one cell must be a straight line connecting two points on the boundary the cell (or  connecting $x$, $y$  to the boundary of the cell).  Hence, any candidate geodesic from $x$ to $y$ consists of a sequence of line segments connecting boundary points of cells to each other and to $x$ and $y$.  Thus, we may suppose that we have the following setup:

Let $P_1,F_1, P_2, F_2, \ldots, F_{k-1}, P_k$ be a sequence of cells of $\calc$ such that
$F_i \subseteq P_i \cap P_{i+1}$ for $i = 1, \ldots, k-1$, $x \in P_1$, $y \in P_k$.  Let $F^\circ$ denote the relative interior of face $F$.  Consider a sequence of points $p_0, \ldots, p_k$ where $p_0 = x$, $p_k =y$ and $p_i \in F_i^\circ$.  The sequence $p_0, \ldots, p_k$ defines a path from $x$ to $y$.  We prove a lemma that characterizes when this path is  shortest among all paths connecting $x$ and $y$ passing through the given sequence of cells and faces.  
 
Let $p$ be a point in an ambient Euclidean space in which the polyhedron $P$ lives and let $F$ be a face of $P$.  There is a well-defined orthogonal projection of $p$ onto the affine space spanned by $F$.  Denote this projection by $\pi_F(p)$.

\begin{lemma}[Zero-Tension Condition]\label{lem:localzero}
Let $P_1,F_1, P_2, F_2, \ldots, F_{k-1}, P_k$ be a sequence of cells of $\calc$ such that
$F_i \subseteq P_i \cap P_{i+1}$ for $i = 1, \ldots, k-1$, $x \in P_1$, $y \in P_k$.  Consider a sequence of points $p_0, \ldots, p_k$ where $p_0 = x$, $p_k =y$ and $p_i \in F_i^\circ$.  The sequence $p_0, \ldots, p_k$ defines a path from $x$ to $y$.  This is the shortest path through this sequence of cells and faces if and only if
\begin{equation}\label{eq:zerotension}
\pi_{F_i}\left(  \frac{ p_i - p_{i-1}}{ || p_i - p_{i-1}||}   \right) = \pi_{F_i}\left(  \frac{ p_{i+1} - p_{i}}{ || p_{i+1} - p_{i}||}   \right)
\end{equation}
for $i = 1, \ldots, k-1$. Here $||x||$ denotes the Euclidean norm.
\end{lemma}

Said concisely, the shortest path will have opposite projections of the unit vector in the direction of $(p_{i-1} - p_i)$ and $p_{i+1} - p_i$.  

\begin{proof}
The shortest path passing through this sequence of orthants, in the given order, will minimize the function:

$$
f(q_1, \ldots, q_k) = ||x - q_1||  + || q_1 - q_2 || + \cdots  + ||q_{k-1} - y||
$$
where $q_i \in F_i$, $i = 1, \ldots, k-1$.    Since we have assumed that $p_i \in F_i^\circ$, such a minimizer must make the projection of the gradient of $f$ onto $F_i^\circ$ be zero.  

Since the derivative of $||x||$ with respect to $x_i$ is $\frac{x_i}{||x||}$ we see that the gradient with respect to $q_i$ is 
$$
 - \pi_{F_i}\left( \frac{q_{i-1} - q_{i}}{||q_{i-1} - q_{i}||}  \right) + \pi_{F_i}\left( \frac{q_{i} - q_{i+1}}{||q_{i} - q_{i+1}||}  \right). 
$$  
This is zero when evaluated at the sequence of points $p_1, \ldots, p_{k-1}$ for all $i$, if and only if Equation (\ref{eq:zerotension}) is satisfied.

Since minimizing the function $f$ is a convex optimization problem by the upcoming Theorem \ref{thm:polish}, the only critical point is the unique minimum.
\end{proof}

Thus we have shown that every geodesic satisfies the Zero-Tension Condition.


\subsection{No Shortcut Condition}\label{sec:noShortcut}
Define a path $\Gamma$ to be a \emph{local geodesic} if there exists some $\epsilon > 0$ so that every subpath of $\Gamma$ of length $\leq \epsilon$ is the shortest path between its endpoints.  The following lemma from \cite{OwenProvan11} (and the more general version in \cite[Chap. II.1, Prop. 1.4]{BridsonHaefliger99}) shows that checking this local condition is sufficient to determine the geodesic.

\begin{lemma}\cite[Lemma 2.1]{OwenProvan11}
In a CAT(0) space, every local geodesic is a geodesic.
\end{lemma}

This implies that we only need to check if a geodesic is locally shortest where it bends.  Since we have shown that a geodesic is contained in a valid cube sequence and satisfies the Zero-Tension Condition, we can assume, without loss of generality, that our path also satisfies these two conditions.

We now need to introduce some notation for analyzing this restricted scenario.  Let $P_1 = C(I_1, M_1), P_2 = C(I_2,M_2), ..., P_k = C(I_k, M_k)$ be the valid cube sequence corresponding to the path, with breakpoint $p_i$ at its intersection with $P_i$ and $P_{i+1}$. 

\begin{lemma} \label{lem:intervalOrigin}
Every maximal cube in $X[P_i,P_{i+1}]$ contains the vertex corresponding to the order ideal $I_i$.
\end{lemma}

\begin{proof}
Let $C(I,M)$ be any maximal cube in the interval $X[P_i, P_{i+1}]$. 
Then we have $I_i \bs M_i \subseteq I \bs M$ and $I \subseteq I_{i+1}$.  We now show
that $I_i \subseteq I$. Let $x \in I_i$. If $x \in M$ then $x \in I$, so assume $x
\notin M$. By the maximality of $M$, $x$ must be comparable with an element $m$ of
$M$. If $x > m$ then $m$ is not maximal in $I_i$, so $m \in I_i \bs M_i \subseteq I
\bs M$, contradicting that $m \in M$. If $x < m$ then $x$ is in $I$ and not maximal
in $I$, so $x \in I - M \subseteq I$.

Since $P_i$ and $P_{i+1}$ are part of a valid cube sequence, $I_{i+1} \bs I_i
\subseteq M_{i+1}$.  So $I \bs I_i \subseteq M_{i+1}$, and hence $I \bs I_i$  is an
anti-chain of maximal elements of $I$.  Since $C(I,M)$ is a maximal cube, $M$ is the
set of maximal elements of $I$.  Thus we have $I \bs I_i \subseteq M$, and therefore
the order ideal $I \bs (I \bs I_i) = I_i$ is a vertex in $C(I,M)$.  
\end{proof}

To determine if the path is locally shortest at breakpoint $p_i$, it is sufficient to determine if there is a shorter path from $p_{i-1}$ to $p_{i+1}$ through the interval $X[P_i,P_{i+1}]$.  Since every maximal cube in $X[P_i,P_{i+1}]$ shares the vertex $I_i$, the interval $X[P_i,P_{i+1}]$ forms what is called a \emph{truncated CAT(0) orthant space}, which we now define.


A $k$-dimensional \emph{orthant} is a cone isometric to $\rr_{\geq 0}^k$.  An orthant space is a collection of orthants with a common origin vertex.  We are interested in orthant spaces which are CAT(0):

\begin{defn} A \emph{d-dimensional CAT(0) orthant space} is denoted by $\mathcal{T} = (V, \Omega)$, where $V$ is the set of \emph{coordinates} and $\Omega$ is a $d$-dimensional flag simplicial complex on $V$.  Every face $F$ of $\Omega$ is associated with the orthant
\[ \calo_F = \rr^{|F|}_{+},
\]
that is, the set of non-negative vectors with components associated with face $F$.  The space $\mathcal{T}$ is the union of all orthants $\{\calo_F: F \in \Omega\}$, with the orthants identified along the subcones of their common coordinates.
\end{defn}

In a \emph{truncated CAT(0) orthant space}, each orthant is replaced with a unit cube of the same dimension.

\begin{cor}
If the cubes $P_i$ and $P_{i+1}$ are part of a valid cube sequence, then $X[P_i,P_{i+1}]$ is a truncated CAT(0) orthant space, whose origin is the vertex corresponding to order ideal $I_i$.
\end{cor}

\begin{proof}
This follows directly from Lemma~\ref{lem:intervalOrigin} and the definition of a truncated CAT(0) orthant space.
\end{proof}

A polynomial time algorithm for finding a shorter path from $p_{i-1}$ to $p_{i+1}$, if one exists, was found for a certain kind of CAT(0) orthant space by Provan and the second author in \cite{OwenProvan11}, and shown to apply to all CAT(0) orthant spaces, including truncated CAT(0) orthant spaces in \cite{MillerOwenProvan10}.


\medskip

Note that $M_i$ and $M_{i+1}$ are the sets of indices of the variable coordinates in the cubes $P_i$ and $P_{i+1}$, respectively. If $z = (z_1, z_2, ..., z_s)$ is any point in the cube complex $X$, and $C$ is some subset of the indices $\{1, 2, ..., s\}$, then let $z |_C$ be the vector with $i$-th coordinate 0 if $i \notin C$ and $z_i$ otherwise.

The following lemma is a restatement in the language of this paper of Theorem 2.5 from \cite{OwenProvan11}, which was shown to generalize to truncated CAT(0) orthant spaces in \cite{MillerOwenProvan10}.

\begin{lemma}[No Shortcut Condition]\label{lem:noShortcut}
A path through a valid cube sequence that satisfies the Zero-Tension Condition is a geodesic if and only if the following condition holds at each breakpoint $p_i$:

For each pair of partitions $A_i \cup B_i$ of $M_i - M_{i+1}$ and $A_{i+1} \cup B_{i+1}$ of $M_{i+1} - M_i$, at least one of which is non-trivial, such that there exists a cube in $X[P_i, P_{i+1}]$ whose variable coordinates have indices $(M_i \cap M_{i+1}) \cup B_i \cup A_{i+1}$, we have

\[
\norm{ (p_i -  p_{i-1}) |_{A_i }} \cdot  \norm{(p_{i+1} - p_i) |_{B_{i+1} } }  \geq  \norm{ ( p_i - p_{i-1}) |_{B_i} } \cdot \norm{ (p_{i +1} - p_i) |_{A_{i+1}} }.
\]

\end{lemma}

The previous inequality can be written more simply (but less symmetrically) as:
\[
\norm{ (\mathbf{1} -  p_{i-1}) |_{A_i }} \cdot  \norm{ p_{i+1} |_{B_{i+1} } }  \geq  \norm{ ( \mathbf{1} - p_{i-1}) |_{B_i} } \cdot \norm{ p_{i +1} |_{A_{i+1}} }
\]
since the $(M_{i+1}-M_i)$-coordinates of $p_i$ equal $0$ 
and its $(M_i-M_{i+1})$-coordinates equal $1$.

Note that $M_i$ and $M_{i+1}$ are antichains and, when we remove their intersection, we obtain a bipartite poset on $(M_i - M_{i+1}) \cup (M_{i+1} - M_i)$. The intermediate cubes of Lemma \ref{lem:noShortcut} that we check correspond to ``intermediate antichains", which contain $M_i \cap M_{i+1}$ and some elements ($B_i$ in the bottom level and $A_{i+1}$ in the top level) of this poset.

We can then use the algorithm for solving the Extension Problem given in \cite[Section~3]{OwenProvan11} for each breakpoint $p_i$, to check whether there is a cube $C$ and partitions $A_i \cup B_i$ and $A_{i+1} \cup B_{i+1}$ which violate the above lemma. If they do, return the cube $C$, which we will add to the cube sequence. (In Lemma \ref{lemma:valid} we will show that the result is still a valid cube sequence.) We will call this the \emph{Shortcut-Checking Algorithm}, which has running time $O(n^3)$, where $n = \max(\dim P_i, \dim P_{i+1}) =  \max(|M_i|, |M_{i+1}|)$.  

The Shortcut-Checking Algorithm determines if there are partitions $A_i \cup B_i$ and $A_{i+1} \cup B_{i+1}$ violating the lemma. To do so, it solves a Minimum Weight Vertex Cover problem on the bipartite graph $G$ induced by the poset on $(M_i - M_{i+1}) \cup (M_{i+1} - M_i)$. The weights of a vertex $j$ in $M_i - M_{i+1}$ and a vertex $k$ in $M_i - M_{i+1}$ are 
 \[
\left(\frac{(p_i-p_{i-1})_j}{||p_i-p_{i-1}||}\right)^2 \quad  \textrm{ and } \quad  \left(\frac{(p_{i+1}-p_i)_k}{||p_{i+1}-p_i||}\right)^2
 \]
respectively. 
A \emph{vertex cover} is a set $S$ of vertices  such that every edge is incident to a vertex in $S$; we seek a vertex cover of minimum weight.
 
One easily checks that partitions violating Lemma \ref{lem:noShortcut} exist if and only if there is a minimum weight vertex cover with total weight less than $1$.  If it does exist, then the sets $A_{i+1}$ and $B_i$ consist of the elements not in the cover. They form a maximal independent set, since the vertex cover has minimum weight.  Thus $(M_{i+1} \cap M_i) \cup A_{i+1} \cup B_i$ is a maximal anti-chain in the poset.  In this case, the Shortcut-Checking Algorithm inserts a cube $C$ into the sequence, with $C = C(I,M)$ where $I = I_{i+1} \bs B_{i+1} = I_i \cup B_i$ and $M =  (M_{i+1} \cap M_i) \cup A_{i+1} \cup B_i$.  Since $M$ is a maximal anti-chain, the cube $C$ inserted by the algorithm is maximal.

\begin{ex}\label{ex:shortcut}
Figure~\ref{fig:blowup} shows the four cubes forming $X[P_3,P_4]$ in the interval $X[v,w]$ in Figure~\ref{fig:validCubeSeq}, along with a segment of the shortest path from $x$ to $y$ contained in the given valid cube sequence.  The two cubes in the valid cube sequence that are in the figure are $C(I_3, M_3)$, where $I_3 = \{1, 2, 5, 6\}$ and $M_3 = \{2,6\}$, and $C(I_4, M_4)$, where $I_4 = \{1, 2, 3, 5, 6, 7\}$ and $M_4 = \{3, 7\}$.  Notice that this segment is not a geodesic in the cube complex, since it would be shorter if it passed through the cube $C(I,M)$.  In this case, the partitions returned by the Shortcut-Checking Algorithm are $A_3 = \{2\}$ and $B_3 = \{6\}$ of $M_3 - M_4 = M_3 = \{2, 6\}$ and $A_4 = \{3\}$ and $B_4 = \{7\}$ of $M_4 - M_3 = M_4 = \{3, 7\}$.  Thus $I = \{1,2,3,5,6\}$ and $M = \{3,6\}$.  The inequality of Lemma \ref{lem:noShortcut} is not satisfied:  

\begin{align*}
\norm{ (\mathbf{1} -  p_2) |_{A_3}} \cdot  \norm{ p_4 |_{B_4} } = \norm{ (0,1,0,0,0,0,0)} \cdot \norm{ (0,0,0,0,0,0,0.45)}= 1 \cdot 0.45 \\
 \leq  1 \cdot 1 =  \norm{ (0,0,0,0,0,1,0)}  \cdot \norm{ (0,0,1,0,0,0,0)} = \norm{ ( \mathbf{1} - p_2) |_{B_3} } \cdot \norm{ p_4 |_{A_4} }
 \end{align*}

\end{ex}

\begin{figure}[h] \label{fig:blowup}
\begin{center}
\includegraphics[width=3in]{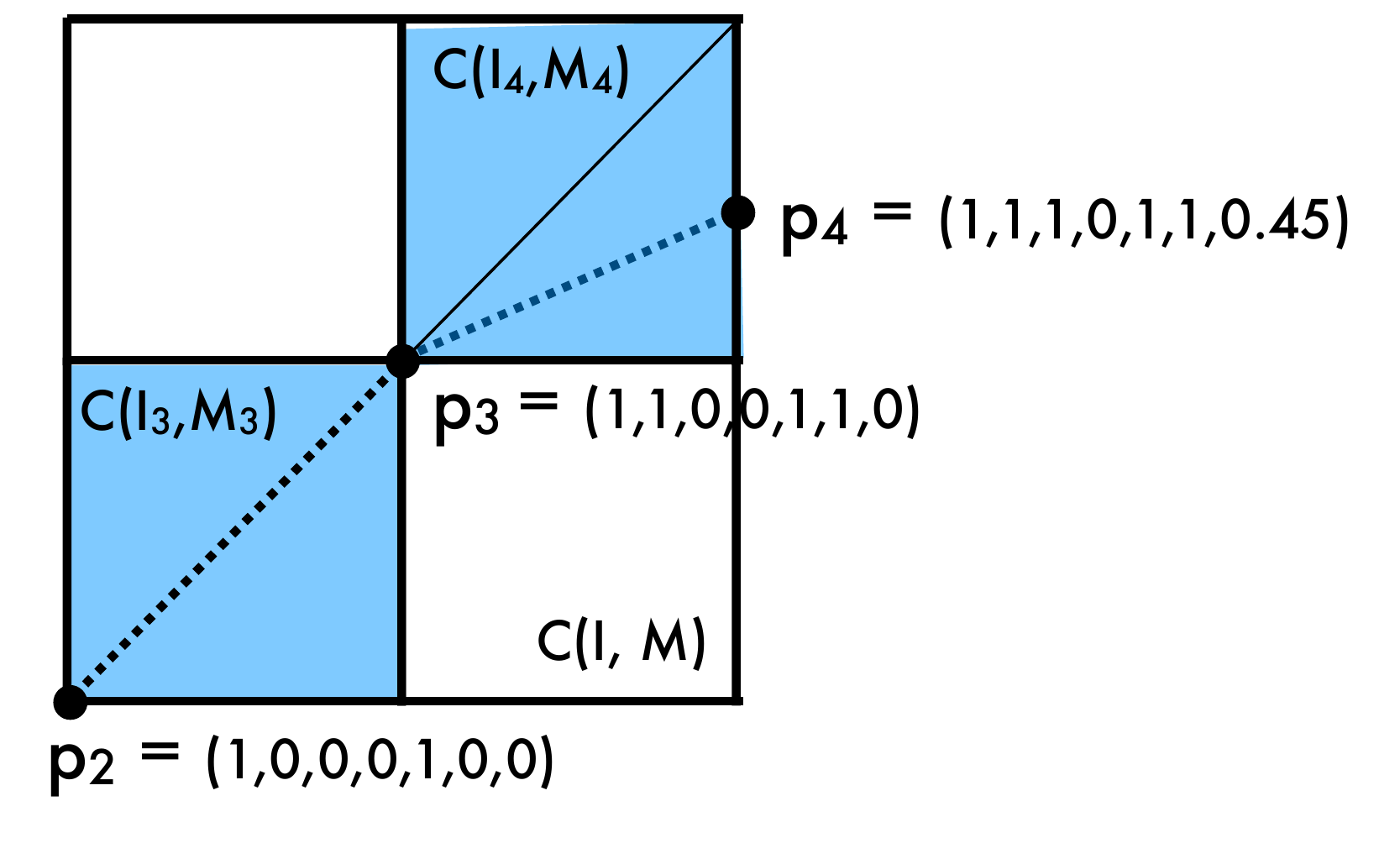} 
\caption{The four cubes forming in $X[P_3,P_4]$ in the interval $X[v,w]$ in Figure~\ref{fig:validCubeSeq}.  The dashed line represents part of the shortest path from $x$ to  $y$ in the given valid cube sequence.  The Shortcut Condition returns cube $C(I,M)$, as explained in Example~\ref{ex:shortcut}. }
\label{fig:blowup}
\end{center}
\end{figure}

This gives us a characterization of geodesics in CAT(0) cube complexes.  

\bigskip

\begin{thm}\label{thm:characterization}
A path from $x$ to $y$ in a CAT(0) cube complex is the geodesic if and only if it
\begin{enumerate}
\item is contained in a valid cube sequence,
\item consists of the union of a finite number of line segments, 
\item satisfies the Zero-Tension Condition, and
\item  satisfies the No Shortcut Condition.
\end{enumerate}
\end{thm}

\begin{proof}
If the path is the geodesic, then by Lemma~\ref{lem:geoInValidCubeSeq} it is contained in a valid cube sequence.  Since a valid cube sequence satisfies the hypotheses of the Zero-Tension Condition by Proposition~\ref{prop:validproperties}.3, the geodesic must also satisfy the Zero-Tension Condition.  Finally, by Lemma~\ref{lem:noShortcut}, we have that the geodesic also satisfies the No Shortcut Condition.    If the path is contained in a valid cube sequence and satisfies both the Zero-Tension Condition and the No Shortcut Condition, then by Lemma~\ref{lem:noShortcut}, it is a geodesic.
\end{proof}


\section{Touring Problems}\label{sec:touring}

Let $S_1, S_2, \ldots, S_k$ be a sequence of regions in $\mathbb{R}^n$, and $x$ and $y$ two points in $\rr^n$.  The touring problem asks for the shortest path starting at $x$ and ending at $y$ and intersecting $S_1$, $S_2$, \ldots, $S_k$, in that order.  In other words, the touring problem asks for points $p_i \in S_i$ that minimize the sum of distances
$$
||x- p_1|| + ||p_1 -  p_2|| + \ldots + ||p_k-y||. 
$$
Note that without the restriction on fixing a prescribed order, this minimization problem includes the geometric traveling salesman problem as a special case and is NP-complete.  Fixing a prescribed ordering, however, and specifying the regions $S_i$ as convex polyhedra of polynomial complexity guarantees that the problem has a polynomial time solution using semidefinite programming \cite{PolishchukMitchell05}.  

\begin{thm} [\cite{PolishchukMitchell05}] \label{thm:polish}
Touring problems of convex polyhedra are convex optimization problems, and can be formulated as semidefinite programs. Furthermore, if there is a polynomial number of polyhedra, and each one has polynomial complexity (in terms of the bit complexity of the input), then the resulting semidefinite program can be solved in polynomial time.
\end{thm}

The importance of the touring problem for us is that the problem of finding the shortest path through a particular valid cube sequence between points $x$ and $y$ is a touring problem.  

\begin{prop}\label{prop:cubetouring}
Let $X_P$ be a CAT(0) cube complex. The problem of computing the shortest path from $x$ to $y$ lying in a given valid cube sequence can be solved by a touring problem which is itself solvable in polynomial time in $|P|$.
\end{prop}

\begin{proof}
Let 
$C(I_1, M_1), C(I_2, M_2), \ldots, C(I_k, M_k)$ be a valid cube sequence.  Using the standard embedding, we can embed the entire cube complex into $\rr^n$.  The shortest path through this cube complex is determined by where the geodesic crosses from cube $C(I_i, M_i)$ into cube $C(I_{i+1}, M_{i+1})$, for each $i = 1, \ldots,  k-1$.  By Proposition \ref{prop:validproperties}, this boundary is the cube $C(I_i, M_i \cap M_{i+1})$.  Thus, we must optimize the function
$$
||x- p_1|| + ||p_1- p_2|| + \ldots + ||p_{k-1}-y||. 
$$
subject to $p_i \in C(I_i, M_i \cap M_{i+1})$.  This is a touring problem with at most $n$ polyhedral regions, each with at most $2n$ facets, where $n$ is the number of elements in the poset $P$.  The coefficients appearing in any defining inequality are all $0$ or $1$.  Hence, this problem is solvable in polynomial time by Theorem \ref{thm:polish}.
\end{proof}

\section{An Algorithm to Compute Geodesics}
\label{sec:geod}

At this point, we have developed all the necessary tools to describe our algorithm for computing geodesics in CAT(0) cube complexes.

\begin{pr}  Given a poset with inconsistent pairs $P$ and two points $x$ and $y$ in the cubical complex $X_P$, find the geodesic path between $x$ and $y$ in $X_P$. 
\end{pr}

By Lemma~\ref{lem:geoInterval}, it is sufficient to find the geodesic between points $x$ and $y$ in the cube complex $X[v,w]$.  We first present the overall algorithm, and then analyze the steps in more detail.

\begin{alg}\label{alg:geodesic}
{\bf [Computing geodesics in a CAT(0) cube complex $X_P$].} \\
Input: A poset with inconsistent pairs $P$ and two points, $x,y$ in $X_P$.  \\
Output: A cube sequence containing the geodesic and the corresponding break points of the geodesic.  \\
\begin{enumerate}
\item Find vertices $v$ and $w$ of $X_P$ such that the geodesic is contained in $X[v,w]$, as described at the beginning of Section \ref{sec:combin.geodesics}. Reroot the complex at $v$, and construct the poset $Q$ such that $X[v,w] \cong X_Q$. 
\item Choose the extended normal cube path as a starting valid cube sequence $P_1, P_2, ..., P_k$.
\item Solve the Touring Problem associated with $P_1, P_2, ..., P_k$ to get the path $\gamma$.
\item Find the smallest valid cube sequence containing $\gamma$, and reset that to be $P_1, P_2, ..., P_k$.
\item  At each breakpoint, check if the No Shortcut condition holds, using the algorithm from Section~\ref{sec:zero}.
\begin{enumerate}
\item
If the No Shortcut condition fails to hold at some $p_i$, then it also returns a new cube $C(I,M)$ that passing through yields a shorter path from $x$ to $y$.  Add $C(I,M)$ to the cube sequence, and re-index.  Go to step 3. 
\item
If the No Shortcut condition holds at all breakpoints $p_1, p_2, ..., p_{k-1}$, then $\gamma$ is the geodesic and the algorithm terminates.
\end{enumerate}\end{enumerate}
\end{alg}

\begin{lemma}
Algorithm \ref{alg:geodesic} terminates in a finite number of steps. 
\end{lemma}

\begin{proof}
This algorithm always finds a strictly shorter path, and this path is always the shortest path through its respective valid cube sequence.  Thus the algorithm cannot return to a previously encountered valid cube sequence.  Since there are only a finite number of cubes in the CAT(0) complex, there are only a finite number of valid cube sequences.
\end{proof}

\subsection{Proof of Correctness of Algorithm \ref{alg:geodesic}}\label{sec:proof} We have already done most of the work to show this. We need the following lemma.

\begin{lemma}\label{lemma:valid}
After applying the Shortcut-Checking algorithm in step 5a, the resulting cube sequence is still valid.
\end{lemma}

\begin{proof}
Let $C(I, M)$ be the maximal cube found by the Shortcut-Checking algorithm at step 5a.  First note that the first and last cubes are the same in all valid cube sequences.  Thus it remains to show that $I_i \subsetneq I \subsetneq I_{i+1}$, $I \bs I_i \subseteq M$, $I_{i+1} \bs I \subseteq M_{i+1}$, and that $M$ is a maximal anti-chain.  The first and last of these hold because the Shortcut-Checking algorithm returns a maximal cube in $X[P_i, P_{i+1}]$.  

We now show that $I \bs I_i \subseteq M$.  Since $I \subseteq I_{i+1}$, we have $I \bs I_i \subset I_{i+1} \bs I_i \subseteq M_{i+1}$, which is the set of all maximal elements of $I_{i+1}$.  Thus $I \bs I_i$ is contained in the set of all maximal elements of $I \in I_{i+1}$, and hence in $M$, which is this set by the maximality of $C(I,M)$.

Finally, since $I_{i+1} \bs I_i \subseteq M_{i+1}$ and $I_i \subset I$, then $I_{i+1} \bs I \subseteq M_{i+1}$.  
\end{proof}

The path returned by the Touring Problem satisfies the Zero-Tension Condition, since it is a geodesic within a valid cube sequence. Since step 5a always returns a valid cube sequence, the final path returned by this algorithm will satisfy the four properties of Theorem~\ref{thm:characterization}, and hence be the geodesic.

 \subsection{Complexity of the Algorithm}
 In analyzing the complexity of the algorithm, we will simply focus on whether each step requires time polynomial or exponential in the number of elements in the relevant poset, ignoring implementation details.  Let $n$ be the number of elements in the input poset $P$, and let $m$ be the number of elements in the poset $Q$ corresponding to the interval $X[v,w]$.
 
To initialize the algorithm, we need to determine $v$ and $w$, reroot the poset at $v$, and then determine the poset $Q$ such that $X[v,w] \cong X_Q$.  We determine $v$ and $w$ by first finding $V$ and $W$, the minimal cubes containing them, by checking which coordinates are strictly between 0 and 1, and then using the procedure explained at the beginning of Section 4.  Since each step can be done by running through the coordinates or elements of $P$, this takes time polynomial in $n$.  

Finding the starting extended normal cube sequence takes time polynomial in $m$, since we do this by repeatedly pruning off all minimal elements of $Q$. To determine the minimal elements there are only ${m \choose 2}$ comparisons to be made, and the number of repetitions is bounded by $m$.

Note that a valid cube sequence contains at most $m$ cubes.   
Thus each step of the iterative part of the algorithm is also polynomial in $m$, as mentioned above, including solving the touring problem for a given valid cube sequence ($O(m^8 \log \frac{1}{\epsilon})$), checking the No Shortcut Condition at each breakpoint ($O(m^3)$), and constructing a new valid cube sequence ($O(m)$).  So each iteration is polynomial in $m$.  

The only part of the algorithm which is potentially non-polynomial is the number of iterations, which is bounded above by the number of valid cube sequences. 

\begin{pr}
Is there a polynomial bound on the number of valid cube sequences that the algorithm visits? 
\end{pr}


\section{Algebraic Complexity of Geodesics}\label{sec:algebra}

In this section, we explore the algebraic complexity of computing geodesics.  According to the  Zero-Tension Condition of Lemma  \ref{lem:localzero}, the geodesic will be the solution to the algebraic system of equations (\ref{eq:zerotension}) in the valid cube sequence in which the geodesic lies.  If the starting and ending points of the geodesic are generic rational numbers, the transition points will be algebraic numbers of a fixed degree that depends only on the poset and the partial linear extension corresponding to the cell-face sequence.  In this section, we explore this algebraic complexity using computational algebra.  This also leads to some (open) classification problems on posets.

For simplicity, we restrict to posets which have the property that the CAT(0) complex $X_{P}$ has exactly one valid cube sequence.  In such a poset $P$, the partial order on maximal antichains 
induced by the distributive lattice $J(P)$ is a linear order.  For this reason, we call these \emph{ascending antichain posets}.

The ascending antichain posets have a unique sequence of maximal cells going from the smallest to the largest cell in the complex $X_{P}$.  These are the posets whose complexes $X_{P}$ arise as valid cube sequences for geodesics in arbitrary CAT(0) complexes.  However, it is not true, in general, that in the cube complex $X_{P}$ of an ascending antichain poset $P$ there exists a pair of points and a geodesic connecting them that passes through the interior of every maximal cell.  We call such ascending antichain posets \emph{bent}, and ascending antichain posets that are not bent are \emph{straight}.

\begin{ex}
The poset with $6$ elements and covering relations
$ 1 \prec 2$, $2 \prec 4$, $2 \prec 5$, $3 \prec 5$, $5 \prec 6$, is an ascending antichain poset.  The maximal antichains in the poset are $\{1,3\}, \{2,3\}, \{3,4\}, \{4,5\}, \{4,6\}$, in that order.  However, this poset is a bent ascending antichain poset because the geodesic from any point in $C(\{1,3\}, \{1,3\})$ to any point in $C(\{1,2,3,4,5,6\}, \{4,6\}) $ does not pass through the relative interior of the maximal cube $C( \{1,2,3,4\}, \{3,4\})$.
\end{ex}

This leads to a problem about posets, a positive solution to which could help rule out valid cube sequences which are globally unnecessary to check.

\begin{pr}
Characterize the straight ascending antichain posets.  More generally, develop a characterization that, given an arbitrary poset $P$, describes the valid cube sequences in which there exists a geodesic  from a point in the bottom cell to a point in the top cell in $X_{P}$ intersecting the interiors of all maximal cells.  
\end{pr}

Suppose that $P$ is a straight ascending antichain poset.  Let $x$ and $y$  be in the bottom and top cells, such that the geodesic from $x$ to $y$ intersects the interior of each maximal cell.  We want to study the algebraic degree of the coordinates of the breakpoints of the geodesic from $x$ to $y$.  This number only depends on $P$ and not on the particular choice of $x$ and $y$, assuming that $x$ and $y$ are generic.  We denote this number $gd(P)$, the \emph{geodesic degree} of $P$.  Beyond the value of the $gd(P)$, we are also interested in knowing the Galois group of the extension if possible.  If the Galois group is complicated, we expect that there should be no especially straightforward algorithm to compute the geodesic besides using the touring problem solution.

First we catalogue some results about the geodesic degree of posets in simple cases.

\begin{prop}\label{prop:2level}
Let $P$ be a straight ascending antichain poset such that every maximal chain has length $1$ or $2$.  Then $gd(P) = 2^n$ for some $n$ and the Galois group of the extension field has the form $\zz_2^n$ for some $n$.  In particular, the coordinates of geodesics lie in a field obtained by adjoining square roots of rational numbers. 
\end{prop}

\begin{proof}
If a
straight ascending antichain poset
 $P$ satisfies the condition that every maximal chain has length $1$ or $2$, then the geodesic in $P$ is solved in linear time by the algorithm in \cite{Owen09}.  This algorithm reduces the computation to computing a line between two points in a Euclidean space after possibly modifying the problem by replacing the points $x$ and $y$ with two new points whose coordinates are at worst square roots of rational numbers.  The intersection of a line connecting two such points with a rational plane lies in the same field, and the Galois group of the normal closure of the field containing these points has the desired form.
\end{proof}

\begin{prop}\label{prop:dim2}
Let $P$ be a straight ascending antichain poset such that the size of the largest antichain is $2$.  Then $gd(P) = 1$.  Thus, in a two dimensional CAT(0) complex, the coordinates  of the break points on the geodesic between two points with rational coordinates are all rational.
\end{prop}

\begin{proof}
If the size of the largest antichain is $2$, then each maximal cell has dimension $2$ or one.  Each intermediate point in a geodesic is either on an edge, or at a vertex of the complex.  If an intermediate point is at a vertex of the complex, it is integral, and hence rational.  If an intermediate point is on an edge of the complex, the zero tension condition implies that the geodesic is a straight line through that point.  Hence the point must have rational coordinates.
\end{proof}

Proposition \ref{prop:dim2} can be extended to handle straight ascending antichain posets where all maximal antichains are of the same size and the break points always lie in the interior of a codimension 1 cell.  In this case, the geodesic must be a straight line, and hence, $gd(P) = 1$.  

At this point, it seems natural to ask whether or not every poset has such a simple description of its geodesics.  That is, is the Galois group of the geodesic always $\zz_2^n$ for some $n$, with the field extension obtained by adjoining square roots of rational numbers?  We were surprised to find that this fails in the simplest example that does not satisfy Propositions \ref{prop:2level} or \ref{prop:dim2}.

\begin{prop}
Let $P$ be the poset on five elements with covering relations $1 \prec 3$, $1 \prec 4$, $2 \prec 5$, $3 \prec 5$.  Then $gd(P) = 8$ and the splitting field which contains all coordinates of the geodesic has Galois group $S_8$.  In particular, the break points on the geodesic cannot be expressed in terms of radicals.
\end{prop}

\begin{proof}
The CAT(0) complex $X(P)$ consists of a three dimensional cube with two squares attached on skew, nonneighboring edges.  In the standard embedding, we must find the geodesic between the fixed points of the form $(a,b,0,0,0)$ and $(1,1,1,c,d)$ passing through points $(1,x,0,0,0)$ and $(1,1,1,y,0)$.  Applying the  Zero-Tension Condition yields the system of two algebraic equations
$$
\frac{b-x}{ \sqrt{(a-1)^2 + (b-x)^2} }  =  \frac{x-1}{\sqrt{(x-1)^2 + y^2 + 1}} 
\quad , 
\quad 
\frac{y}{ \sqrt{(x-1)^2 + y^2 + 1}} = \frac{y-c}{\sqrt{ (y-c)^2 + d^2 }}.
$$ 
Squaring both sides and clearing denominators yields a polynomial system.  The following Singular \cite{DGPS} computation shows that there are generically eight solutions to this system of equations.

\begin{verbatim}
ring R = (0,a,b,c,d),(x,y), dp;
ideal i =
(b-x)^2*((x-1)^2 + y^2 + 1) - (x-1)^2*((a-1)^2 + (b-x)^2),
y^2*((y-c)^2 + d^2) - (y-c)^2*( (x-1)^2 + y^2 + 1);
ideal j = std(i);
degree(j);
\end{verbatim}

Making a specific choice of values for $a,b,c$, and $d$, we compute a lexicographic Gr\"obner basis for the ideal, which contains a polynomial of degree $8$ in the variable $y$.  Plugging this polynomial into Maple we verify that this polynomial, and hence the polynomial system, has Galois group $S_8$.  This implies that the system is not solvable by radicals.  Since the order of the Galois group is upper-semicontinuous in the parameters, we see that this system is not solvable in radicals for generic choices of the parameters.
\end{proof}

These calculations show that there will be no simple closed form formulas for the breakpoints in the geodesic, and suggest that any algorithm for computing geodesics should depend on iteration in some way.


\bibliographystyle{plain}

\bibliography{geodesic}

\end{document}